\documentclass[12pt,a4paper]{article}
\usepackage{amssymb}
\usepackage{amsmath}
\usepackage{amsthm}
\usepackage{mathtools}
\usepackage{graphicx}
\usepackage{hyperref}
\newtheorem{thm}{Theorem}
\newtheorem{cor}[thm]{Corollary}

\newtheorem{prop}[thm]{Proposition}
\newtheorem{defn}{Definition}
\newtheorem*{rem}{Remark}
\newtheorem*{examp}{Example}
\begin{document}
\title{Singularities of the Moduli Space of $n$ Unordered Points on the Riemann Sphere}
\author{Yue Wu and Bin Xu$^\dagger$}
\date{}

\maketitle

\begin{abstract}
We classify the finite groups associated to the singularities of the moduli space of $n\geq 5$ unordered points on the Riemann sphere. We also realize the classification  by an algorithm.
\end{abstract}
\footnote{
$^\dagger$
The second author is supported in part by the National Natural Science Foundation of China (Grant No. 11571330) and the Fundamental Research Funds for the Central
Universities.}

\tableofcontents
\newpage
\section{Introduction}

Let $\mathfrak{M_{g,\:n}}$ denote the moduli space of isomorphism classes of compact Riemann surfaces of genus $g$ with $n$ unordered marked points.
It is well-known that $\mathfrak{M_{g,\:n}}$ is both a complex orbifold and an irreducible quasi-projective variety of dimension $3g-3+n$, where $n\ge3$ if $g=0$, $n\ge1$ if $g=1$ and $n\ge0$ if $g\ge2$ \cite{De-Mu}.

In section \ref{orbifold}, using elementary method we show that $\mathfrak{M_{0,\:n}}$ is a complex orbifold of dimension $n-3$ when $n\ge4$.
Moreover, we show that when $n\ge5$, $\mathfrak{M_{0,\:n}}$ is the quotient space of an $S_n$ action upon a particular Zariski open set $K_n \subseteq\mathbb{C}^{n-3}$,
and the stabilizers of this $S_n$ action correspond to all the subgroups of
\[{\rm PSL}(2,\,{\Bbb C})=\{\text{M\" obius transformations}\}\]
which are stabilizers of sets that consist of $n$ distinct points on the extended complex plane ${\widehat {\Bbb C}}={\Bbb C}\cup\{\infty\}$ (for a set $\alpha$ consists of $n$ distinct points on the extended complex plane, its stabilizer ${\mathcal A}_\alpha$ refers to the group consists of all the M\" obius transformations which leave $\alpha$ invariant). 
Thus the classification of the orbifold singularities of $\mathfrak{M_{0,\:n}}$ when $n\ge5$ is reduced to the following problem:
\begin{quote}
{\it Classify all the stabilizers of sets that consist of $n$ distinct points on the extended complex plane ${\widehat {\Bbb C}}$.}
\end{quote}
We note that such a stabilizer is a finite subgroup of ${\rm PSL}(2,\,{\Bbb C})$ as $n\geq5$. We also mention the following references for the widely known classical cases of $n=4,\,5$ and $6$ about the stabilizers: \cite{Wiman} and \cite[Chapter 8 ]{Dol} for the cases of $n=4$ and $5$, and
\cite{Bolza} for the case of $n=6$.

It is well known that there are three types of non-trivial finite subgroups of ${\rm PSL}(2,\,{\Bbb C})$ up to conjugacy: the polyhedral groups, the dihedral groups and the cyclic groups. Let $G$ be one of these groups. Then $\alpha$ is decomposed into finitely many orbits of $G$ if ${\mathcal A}_\alpha$ coincides with $G$. However, we note that the converse of this assertion is {\it not} true in general. We thus introduce a notion called the \emph{component index} of $\alpha$ to describe the number of orbit(s) in each type in this decomposition and its definition is left to Subsection \ref{subsec:subsets}.

To tackle our problem, first we give equivalent conditions concerning the constitution of $\alpha$ for its stabilizer ${\mathcal A}_\alpha$ to coincide with $G$. Second we discuss the range of possible values for both the component index and the cardinality of $\alpha$ when ${\mathcal A}_\alpha$ coincides with $G$. Last we summarize these results to a complete answer to the problem and the classification of the orbifold singularities of $\mathfrak{M_{0,\:n}}$ when $n\geq 5$. We express the answer by an algorithm written in pseudo-code in Section \ref{algorithm} and C programming language in Section \ref{code}. We also mention an interesting observation as follows:
\begin{quote}
{\it For each finite subgroup $G$ of $\text{\rm PSL}(2,\:\mathbb C)$, there exists an integer $n\geq 5$ and some orbifold singularity of $\mathfrak{M_{0,\:n}}$ with stabilizer $G$.}
\end{quote}

The organization of this manuscript is as follows. Section \ref{main} is the main body of this paper, presenting all the major steps in tackling our problem but leaving out some proof details to the following sections.
Section \ref{orbifold} discusses the correspondence between the orbifold stabilizers of $\mathfrak{M_{0,\:n}}$ and the finite subgroups of ${\rm PSL}(2,\,{\Bbb C})$ which are stabilizers for sets that consist of $n$ distinct points. It also gives an explicit representation of these stabilizers. 
Section \ref{cardinality} discusses the range of possible values for both the component index and the cardinality of $\alpha$ by giving concrete examples for the cases when ${\mathcal A}_\alpha$ is dihedral, cyclic and trivial. The results are listed in Subsection \ref{subsec:subsets} and \ref{algorithm}.
Section \ref{code} is the code for classification written in C programming language.

\section{Main Results}\label{main}

Subsection \ref{subsec: orbifold} shows the correspondence between the orbifold stabilizers of $\mathfrak{M_{0,\:n}}$ and the finite subgroups of ${\rm PSL}(2,\,{\Bbb C})$ which are stabilizers for sets consisting of $n$ distinct points on the extended complex plane (proof is left to Section \ref{orbifold}). Subsection \ref{subsec:subsets} discusses the range of possible values for both the component index and the cardinality of $\alpha$ (some proofs are left to Section \ref{cardinality}). Subsection \ref{algorithm} presents the final answer to our problem: the classification algorithm of the orbifold singularities of $\mathfrak{M_{0,\:n}}$ when $n\geq 5$.

\subsection{Singularities of ${\frak {M_{0,n}}}$ and Stabilizers of Sets of $n$ Points on $\widehat {\Bbb C}$} \label{subsec: orbifold}

Let $C$ and $C'$ be two compact Riemann surfaces of genus $g$ and 
$
\{p_1, p_2, \cdots p_n\}\subseteq C
$, $
\{p'_1, p'_2, \cdots p'_n\}\subseteq C'.
$
We say that $(C, \{p_1, p_2, \cdots p_n\})$ and $(C', \{p'_1, p'_2, \cdots p'_n\})$ are isomorphic if there exists some biholomorphic map $f: C\to C'$ such that
$
f(\{p_1, p_2, \cdots p_n\})=\{p'_1, p'_2, \cdots p'_n\}.
$
\begin{defn}[\cite{Mu-Pe}]
The moduli space $\mathfrak{M_{g,\:n}}$ is the set of isomorphism classes of compact Riemann surfaces of genus $g$ with $n$ unordered marked points.
\end{defn}
In Section \ref{orbif} using elementary methods we shall prove that
\begin{thm}
$\mathfrak{M_{0,\:n}}=K_n/G_n$ for $n\ge4$, where $K_n$ is the Zariski open set
$$
K_n=\{\boldsymbol{\lambda}
=(\lambda_1, \lambda_2, \cdots \lambda_{n-3})\in\mathbb C^{n-3}
\:|\:\lambda_i\ne0, 1,\:\lambda_i\ne\lambda_j,\:
\forall i, j=1, 2, \cdots n-3,\:i\ne j\}
$$
and $G_n$ is a finite group of birational transformations of $\mathbb C^{n-3}$, whose restrictions to $K_n$ are isomorphisms on $K_n$.
\end{thm}
\begin{cor}\label{34}
$\mathfrak{M_{0,\:n}}=K_n/G_n$ is a complex orbifold of dimension $n-3$ when $n\ge4$.
\end{cor}
\begin{rem}
$\mathfrak{M_{0,\:n}}$ is a single point if $n\le3$, and $\mathfrak{M_{0,\:4}}$ coincides with the moduli space of elliptic curves.
\end{rem}

In Section \ref{grou} we shall find that
\begin{thm}
$G_n$ is isomorphic to the symmetric group $S_n$ when $n\ge5$.
\end{thm}
\begin{cor}
The moduli space $\mathfrak{M_{0,\:n}}$ is the quotient space of a Zariski open set $K_n$ of $\mathbb C^{n-3}$ by an $S_n$ action  when $n\ge5$.
\end{cor}

For $\boldsymbol {\lambda}\in K_n$, set
$
[\boldsymbol {\lambda}]=\{0, 1, \infty, \lambda_1, \cdots \lambda_{n-3}\}
$, $
G_{\boldsymbol\lambda}=\{g_{\sigma}\in G_n\:|\:g_\sigma(\boldsymbol\lambda)=\boldsymbol\lambda\}.
$
Given a subset $\alpha$ of the extended complex plane, let $\mathcal{A}_{\alpha}$ denote its stabilizer, i.e., the group consists of all the M\" obius transformations which leave $\alpha$ invariant. In Section \ref{sing} we shall prove that
\begin{thm}
$G_{\boldsymbol \lambda}$ is isomorphic to $\mathcal{A}_{[\boldsymbol {\lambda}]}$ when $n\ge5$.
\end{thm}
Thus the classification of the orbifold singularities of $\mathfrak{M_{0,\:n}}$ when $n\ge5$ is reduced to the classification of all the stabilizers of sets that consist of $n$ distinct points on the extended complex plane ${\widehat {\Bbb C}}$.

\subsection{Finite Subsets with Their Stabilizers Prescribed}
\label{subsec:subsets}

Given a finite subset $\alpha$ of the extended complex plane, it is obvious that its stabilizer is finite when $|\alpha|\ge3$. There are three types of non-trivial finite groups consist of M\"obius transformations: the polyhedral groups, the dihedral groups and the cyclic groups. Let $G$ be one of these groups. We will give a classification of finite subsets whose stabilizer equals $G$. We shall work on the Riemann sphere
$$S=\{(x_1,x_2,x_3)\in {\Bbb R}^3:\, x_1^2+x_2^2+x_3^2=1\}$$
in the polyhedral cases and on $\widehat{\Bbb C}$ in both the dihedral and the cyclic cases.

Suppose that $G$ is isomorphic to an icosahedral group (or octahedral group). Then it fixes a regular dodecahedron (resp. cube) whose center is the origin of the Riemann sphere $S$. Let $V$, $F$ and $E$ denote its vertices, the projections of the central points of its faces on $S$ and the projections of the middle points of its edges on $S$ respectively, with the origin being the central of the projection. For any $X\in S \backslash (V\cup F\cup E)$, define $B(X)$ as the orbit of $X$ under $G$.
\begin{defn}
For any subset $\alpha\subseteq S$ that is a finite union of the orbits of
$G$ ($\simeq A_5$ or $S_4$), define its \emph{(icosahedral or octahedral) component index} as
$(\nu, \mu, \epsilon, k)\in\{0, 1\}^3\times\mathbb N$, where $\nu, \mu, \epsilon, k$ indicate the number of orbit(s) in the $V$, $F$, $E$ and $B$ type, respectively.
\end{defn}
If the stabilizer of an finite subset $\alpha$ equals $G$, $\alpha$ must be a finite union of the orbits of $G$. Conversely, if $\alpha$ is a finite union of the orbits of $G$, its stabilizer $\mathcal{A}_{\alpha}$ includes $G$ as a subgroup. Since $\mathcal{A}_{\alpha}$ is a finite subgroup of ${\rm PSL}(2,\,{\Bbb C})$, $\mathcal{A}_{\alpha}$ equals $G$. Thus we have
\begin{thm}
For any finite subset $\alpha\subseteq S$, its stabilizer $\mathcal{A}_{\alpha}$ equals $G$ ($\simeq A_5$ or $S_4$) if and only if $\alpha$ is a finite union of certain elements in $\{V,\:F,\:E\}\cup\{B(X)|X\in S \backslash (V\cup F\cup E)\}$.
\end{thm}
\begin{cor}
For each non-zero element $(\nu, \mu, \epsilon, k)\in\{0, 1\}^3\times\mathbb N$, there exist some finite subsets $\alpha,\:\beta\subseteq S$ such that $\mathcal{A}_{\alpha}\simeq A_5$ and $\mathcal{A}_{\beta}\simeq S_4$, both with the component index $(\nu, \mu, \epsilon, k)$. Furthermore, we have
$$
\{|\alpha|<\infty\:|\:\mathcal{A}_{\alpha}\simeq A_5\}=
\{x+60k\:|\:x=0, 12, 20, 30, 32, 42, 50, 62,\:k\in\mathbb N,\:x^2+k^2\ne0\};
$$
$$
\{|\beta|<\infty\:|\:\mathcal{A}_{\alpha}\simeq S_4\}=
\{x+24k\:|\:x=0, 6, 8, 12, 14, 18, 20, 26,\:k\in\mathbb N,\:x^2+k^2\ne0\}.
$$
\end{cor}

Suppose that $G$ is isomorphic to a tetrahedral group. Then it fixes two regular tetrahedrons whose centers are the origin of the Riemann sphere $S$. Choose one of the tetrahedrons (this choice will not affect our conclusions). Let $V$, $F$ and $E$ denote its vertices, the projections of the central points of its faces on $S$ and the projections of the middle points of its edges on $S$ respectively, with the origin being the central of the projection. For any $X\in S \backslash (V\cup F\cup E)$, define $B(X)$ as the orbit of $X$ under $G$. Notice that $V$ and $F$ have the same cardinality and thus are considered to be in the same type.
\begin{defn}
For any subset $\alpha\subseteq S$ that is a finite union of the orbits of
$G$ ($\simeq A_4$), define its \emph{(tetrahedral) component index} as $(\nu, \epsilon, k)\in\{0, 1, 2\}\times\{0, 1\}\times\mathbb N$, where $\nu, \epsilon, k$ indicate the number of orbit(s) in the $V(F)$, $E$ and $B$ type, respectively.
\end{defn}
If the stabilizer of an finite subset $\alpha$ equals $G$, $\alpha$ must be a finite union of the orbits of $G$. Conversely, if $\alpha$ is a finite union of the orbits of $G$, its stabilizer $\mathcal{A}_{\alpha}$ includes $G$ as a subgroup. Since $\mathcal{A}_{\alpha}$ is a finite subgroup of ${\rm PSL}(2,\,{\Bbb C})$, $\mathcal{A}_{\alpha}$ could be isomorphic to $A_5$, $S_4$ or $A_4$.
\begin{description}
\item[CASE 1: $\mathcal{A}_{\alpha}$ is Isomorphic to $A_5$.]
If $\mathcal{A}_{\alpha}$ is isomorphic to $A_5$, then $\mathcal{A}_{\alpha}$ fixes some regular dodecahedron $I$ whose center is the origin. Their relative positions are shown in Figure \ref{A_5, A_4}.
\begin{figure}[!h]
\centering
\includegraphics[width=3in]{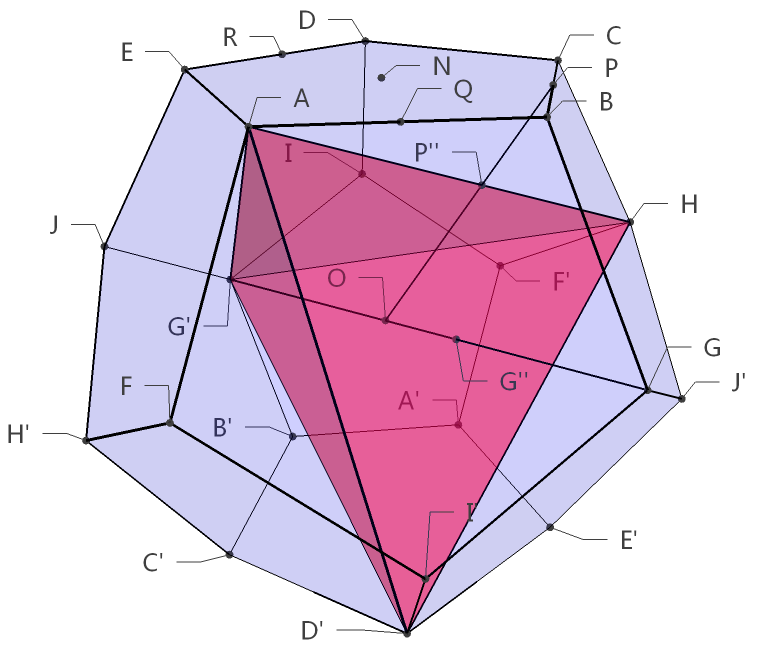}
\caption{The dodecahedron $I$ fixed by $\mathcal{A}_{\alpha}$ and the tetrahedron $T$ fixed by $G$.}
\label{A_5, A_4}
\end{figure}

It is easy to check that ($V_I$, $F_I$, $E_I$ and $B_I(X)$ denote the orbits under the icosahedral group here)
$$
V_I=V\cup F\cup B(B), F_I=B(N), E_I=E\cup B(Q)\cup B(R),
$$
where $N$ denotes the center of the pentagon $ABCDE$, and $Q$, $R$ denote the midpoints of $AB$, $DE$ respectively, and
$$
B_I(X)=B(X)\cup B(g(X))\cup B(g^2(X))\cup B(g^3(X))\cup B(g^4(X)),
$$
where $g$ denotes the rotation of order $5$ that fixes the pentagon $ABCDE$ for each $X\in S \backslash (V_I\cup F_I\cup E_I)$.
\item[CASE 2: $\mathcal{A}_{\alpha}$ Is isomorphic to $S_4$.]
If $\mathcal{A}_{\alpha}$ is isomorphic to $S_4$, then $\mathcal{A}_{\alpha}$ fixes some regular cube $O$ whose center is the origin. Their relative positions are shown in Figure \ref{S_4, A_4}.
\begin{figure}[!h]
\centering
\includegraphics[width=2in]{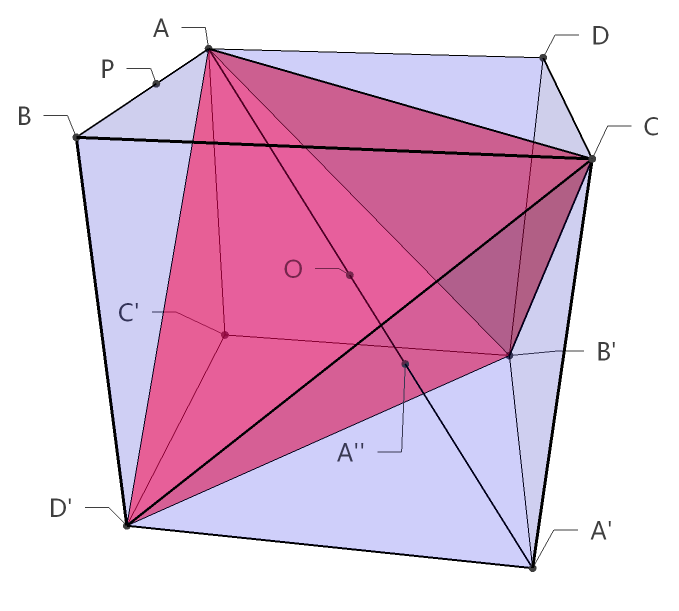}
\caption{The cube $O$ fixed by $\mathcal{A}_{\alpha}$ and the tetrahedron $T$ fixed by $G$.}
\label{S_4, A_4}
\end{figure}

It is easy to check that ($V_O$, $F_O$, $E_O$ and $B_O(X)$ denote the orbits under the octahedral group here)
$$
V_O=V\cup F,\:F_O=E,\:E_O=B(P)
$$
where $P$ denotes the midpoint of $AB$ and
$$
B_O(X)=B(X)\cup B(g(X))
$$
where $g$ denotes the rotation of order $4$ that fixes the the square $ABCD$ for each $X\in S \backslash (V_O\cup F_O\cup E_O)$.
\end{description}
\begin{thm}
For any finite subset $\alpha\subseteq S$, its stabilizer $\mathcal{A}_{\alpha}$ equals $G$ ($\simeq A_4$) if and only if all of the three claims are true:
\begin{enumerate}
\item $\alpha$ is a union of certain elements in
$$
\{V,\:F,\:E\}\cup\{B(X)|X\in S \backslash (V\cup F\cup E)\};
$$
\item $\alpha$ is NOT a union of certain elements in
$$
\{V_I,\:F_I,\:E_I\}\cup\{B_I(X)|X\in S \backslash (V_I\cup F_I\cup E_I)\};
$$
\item  $\alpha$ is NOT a union of certain elements in
$$
\{V_O,\:F_O,\:E_O\}\cup\{B_O(X)|X\in S \backslash (V_O\cup F_O\cup E_O)\}.
$$
\end{enumerate}
\end{thm}
From the discussion above we know that for each non-zero element $(\nu,\:\epsilon,\:k)\in\{0,\:1,\:2\}\times\{0,\:1\}\times\mathbb N$,
$(\nu, \epsilon, k)\ne(2, 0, 0)$, $(0, 1, 0)$ or $(2, 1, 0)$,
there exists some finite subset $\alpha\subseteq S$ such that $\mathcal{A}_{\alpha}\simeq A_4$ with tetrahedral component index $(\nu, \epsilon, k)$.
When the tetrahedral component index of $\alpha$ is $(2, 0, 0)$, $(0, 1, 0)$ or $(2, 1, 0)$, $\mathcal{A}_{\alpha}$ is isomorphic to $S_4$.

\begin{thm}
The existence of $\alpha$ such that $\mathcal{A}_{\alpha}\simeq S_4$ with arbitrary element in $\{0, 1, 2\}\times\{0, 1\}\times\mathbb N$ as its tetrahedral component index is shown in the table below, along with all possible cardinality of $\alpha$.
$$
\begin{array}{c|c|c|c|c|c|c}
\mathrm{Component\,Index}
&(1, 0, 0)&(2, 0, 0)&(0, 1, 0)&(1, 1, 0)&(2, 1, 0)
&(\nu, \epsilon, k),\:k\ge1\\\hline
\mathrm{Existence\:}(|\alpha|)
&\mathrm{Yea\:}(4)&\mathrm{Nay}&\mathrm{Nay}&\mathrm{Yea\:}(10)&\mathrm{Nay}
&\mathrm{Yea\:}(4\nu+6\epsilon+12k)\\
\end{array}
$$
$$
\{|\alpha|<\infty\:|\:\mathcal{A}_{\alpha}\simeq A_4\}=
\{x+12k\:|\:x=0, 4, 6, 8, 10, 14,\:k\in\mathbb N^+,\}
\cup\{4, 10\}.
$$
\end{thm}

Suppose that G is isomorphic to a dihedral group. In this case we go back to work on $\widehat{\mathbb C}$. Without loss of generality set $(n\ge2)$
$$
G=\langle z\mapsto e^{\frac{1}{n}2\pi i}z,\:
          z\mapsto \frac{1}{z}              \rangle\simeq D_n.
$$
There are three types of orbits:
$$
V=\{0, \infty\},
$$
$$
A_n=\{e^{\frac{k}{n}2\pi i}|k\in\mathbb{Z}\},\:
B_n=\{e^{\frac{2k+1}{2n}2\pi i}|k\in\mathbb{Z}\}
$$
($A_n$ and $B_n$ have the same cardinality and are thus considered to be in the same type), and
$$
C_n(z)=\{ze^{\frac{k}{n}2\pi i}|k\in\mathbb{Z}\}\cup\{z^{-1}e^{\frac{k}{n}2\pi i}|k\in\mathbb{Z}\}
$$
for $z\in{\mathbb C}^*\backslash\{e^{\frac{l}{2n}2\pi i}|l\in\mathbb{Z}\}$ when $n\ge3$. When $n=2$, there are only two types of orbits since $V$, $A_2$ and $B_2$ all have the same cardinality $2$.
\begin{defn}
For any subset $\alpha\subseteq S$ that is a finite union of the orbits of $G$ ($\simeq D_n$, $n\ge3$), define its \emph{(dihedral-$n$) component index} as
$(\nu, \epsilon, k)\in\{0, 1\}\times\{0, 1, 2\}\times\mathbb N$, where $\nu, \epsilon, k$ indicate the number of orbit(s) in the $V$, $A_n(B_n)$ and $C_n$ type, respectively.

For any subset $\alpha\subseteq S$ that is a finite union of the orbits of $G$ ($\simeq K_4$), define its \emph{(dihedral-$2$) component index} as
$(\nu, k)\in\{0, 1, 2, 3\}\times\mathbb N$, where $\nu, k$ indicate the number of orbit(s) in the $V(A_2,\:B_2)$ and $C_2$ type, respectively.
\end{defn}
If the stabilizer of an finite subset $\alpha$ equals $G$, $\alpha$ must be a finite union of the orbits of $G$. Conversely, if $\alpha$ is a finite union of the orbits of $G$, its stabilizer $\mathcal{A}_{\alpha}$ includes $G$ as a subgroup. Since $\mathcal{A}_{\alpha}$ is itself a finite group consisting of M\"obius transformations, $\mathcal{A}_{\alpha}$ could be isomorphic to $A_5$ (when $n=2$, $3$ or $5$), $S_4$ (when $n=2$, $3$ or $4$), $A_4$ (when $n=2$) or $D_{pn}$ for $p\in\mathbb Z^+$. Since the $A_5,\:S_4,\:A_4$ cases have been explored, we shall discuss the case when $\mathcal{A}_{\alpha}$ is isomorphic to $D_{pn},\:p\ge2$.
Set $\mathcal{A}_{\alpha}=\langle \rho,\pi|(\rho)^{pn}=(\pi)^2=(\rho\pi)^2=e\rangle\simeq D_{pn}$.\\

\begin{description}
\item[CASE 1: $G\simeq D_n$, $n\ge3$.]
This is simple since there must be certain $m\in\mathbb Z$ such that $\rho^m(z)=e^{\frac{1}{n}2\pi i}z$. Thus
$$
\mathcal{A}_{\alpha}=\langle z\mapsto e^{\frac{1}{pn}2\pi i}z,\:z\mapsto\frac{1}{z} \rangle
$$
and the relations of the orbits of $G$ and $\mathcal{A}_{\alpha}$ are obvious.
\item[CASE 2: $G\simeq K_4$.]
This is relatively complicated for there are three possible relationships between $G$ and $\mathcal{A}_{\alpha}$, i.e.,
$$
\rho^p(z)=-z;\:\rho^p(z)=\frac{1}{z};\:\rho^p(z)=\frac{-1}{z}.
$$
It is easy to check that the orbits of $\mathcal{A}_{\alpha}$ are
$$
V,\:A_{2p},\:B_{2p},\:C_{2p}(z),\:z\in{\mathbb C}^*\backslash\{e^{\frac{l}{4p}2\pi i}|l\in\mathbb{Z}\};
$$
$$
\phi(V,)\:\phi(A_{2p}),\:\phi(B_{2p}),\:\phi(C_{2p}(z)),\:z\in{\mathbb C}^*\backslash\{e^{\frac{l}{4p}2\pi i}|l\in\mathbb{Z}\};
$$
$$
\psi(V,)\:\psi(A_{2p}),\:\psi(B_{2p}),\:\psi(C_{2p}(z)),\:z\in{\mathbb C}^*\backslash\{e^{\frac{l}{4p}2\pi i}|l\in\mathbb{Z}\};
$$
respectively, where
$$
 \phi(z)=\frac{z-1}{z+1},\: \psi(z)=\frac{z+i}{iz+1}.
$$
\end{description}
\begin{thm}\label{tD_n}
For any finite subset $\alpha\subseteq S$, $|\alpha|\ge3$, its stabilizer $\mathcal{A}_{\alpha}$ equals $G$ ($\simeq D_n)$ $(n\ge2)$ if and only if all of the five claims are true:
\begin{enumerate}
\item $\alpha$ is a union of certain elements in
$$
\{V,\:A_n,\:B_n\}\cup\{C_n(z)|z\in{\mathbb C}^*\backslash\{e^{\frac{l}{2n}2\pi i}|l\in\mathbb{Z}\}\};
$$
\item $\alpha$ is NOT a union of certain elements in
$$
\{V,\:A_{pn},\:B_{pn}\}\cup\{C_{pn}(z)|z\in{\mathbb C}^*\backslash\{e^{\frac{l}{2np}2\pi i}|l\in\mathbb{Z}\}\},\:p\ge2;
$$
\item when $n=2$, $\alpha$ is NOT a union of certain elements in
$$
\{\phi(V),\:\phi(A_{2p}),\:\phi(B_{2p})\}\cup\{\phi(C_{2p}(z))|z\in\mathbb C^*\backslash\{e^{\frac{l}{4p}2\pi i}|l\in\mathbb{Z}\}\},\:p\ge2;
$$
\item when $n=2$, $\alpha$ is NOT a union of certain elements in
$$
\{\psi(V),\:\psi(A_{2p}),\:\psi(B_{2p})\}\cup\{\psi(C_{2p}(z))|z\in\mathbb C^*\backslash\{e^{\frac{l}{4p}2\pi i}|l\in\mathbb{Z}\}\},\:p\ge2;
$$
\item $\alpha$ is NOT in the icosahedral, octahedral or tetrahedral case,
\end{enumerate}
where $\phi(z)=\frac{z-1}{z+1}$ and $\psi(z)=\frac{z+i}{iz+1}$.
\end{thm}
The discussion of the range of possible values for the component index in this case is complicated and left in Section \ref{D_n}. The results are
\begin{thm}\label{cD_n}
The existence of $\alpha$ such that $\mathcal{A}_{\alpha}\simeq D_n$, $n\ge2$ with arbitrary element in $\{0, 1\}\times\{0, 1, 2\}\times\mathbb N$ or $\{0, 1, 2, 3\}\times\mathbb N^+$ as its dihedral-$n$ component index is shown in the tables below, along with all possible cardinality of $\alpha$ in each circumstances.
$$
\begin{array}{c|c|c|c|c|c}
\mathrm{Component\:Index}&(0, 1, 0)&(1, 1, 0)&(0, 2, 0)&(1, 2, 0)& (\nu, \epsilon, k),\:k\ge1\\\hline
n\ge3,\;n\ne4&\mathrm{Yea\:}(n)&\mathrm{Yea\:}(n+2)&\mathrm{Nay}&\mathrm{Nay}&\mathrm{Yea\:}(2\nu+n(\epsilon+2k))\\\hline
n=4&\mathrm{Yea\:}(n)&\mathrm{Nay}&\mathrm{Nay}&\mathrm{Nay}&\mathrm{Yea\:}(2\nu+n(\epsilon+2k))
\end{array}
$$
$$
\begin{array}{c|c|c|c|c}
\mathrm{Component\:Index}&(1, 0)&(2, 0)&(3, 0)&(\nu, k),\:k\ge1\\\hline
n=2&\mathrm{Nay}&\mathrm{Nay}&\mathrm{Nay}&\mathrm{Yea\:}(2\nu+4k)\\
\end{array}
$$
$$
\{|\alpha|<\infty\:|\:\mathcal{A}_{\alpha}\simeq D_n\}=
\{x+nk\:|\:x=0, 2,\:k\ge1,\},\:n\ge3,\:n\ne4.
$$
$$
\{|\alpha|<\infty\:|\:\mathcal{A}_{\alpha}\simeq D_4\}=
\{x+4k\:|\:x=0, 2,\:k\ge2,\}\cup\{4\},\:n=4.
$$
$$
\{|\alpha|<\infty\:|\:\mathcal{A}_{\alpha}\simeq K_4\}=
\{2k\:|\:k\ge2,\},\:n=2.
$$
\end{thm}

Suppose that G is isomorphic to a cyclic group. Without loss of generality set $(n\ge2)$
$$
G=\langle
z\mapsto e^{\frac{1}{n}2\pi i}z \rangle\simeq\mathbb{Z}_n.
$$
There are two types of orbits:
$$
N=\{\infty\},\:S=\{0\},:
C_n(z)=\{e^{\frac{k}{n}2\pi i}z|k\in\mathbb{Z}\}
$$
for $z\in\mathbb C^*$ ($N$ and $S$ are in the same type).
\begin{defn}
For any subset $\alpha\subseteq S$ that is a finite union of the orbits of $G$ ($\simeq \mathbb{Z}_n$), define its \emph{(cyclic-$n$) component index} as
$(\nu, k)\in\{0, 1, 2\}\times\mathbb N$, where $\nu, k$ indicate the number of orbit(s) in the $N(S)$ and $C_n$ type, respectively.
\end{defn}
If the stabilizer of an finite subset $\alpha$ equals $G$, $\alpha$ must be a finite union of the orbits of $G$. Conversely, if $\alpha$ is a finite union of the orbits of $G$, its stabilizer $\mathcal{A}_{\alpha}$ includes $G$ as a subgroup. Since $\mathcal{A}_{\alpha}$ is itself a finite group consisting of M\"obius transformations, $\mathcal{A}_{\alpha}$ could be isomorphic to $A_5$ (when $n=2$, $3$ or $5$), $S_4$ (when $n=2$, $3$ or $4$), $A_4$ (when $n=2$ or $3$), $D_q$ for $q\ge2$ (when $n=2$), $D_{pn}$, $\mathbb{Z}_{pn}$ for $p\ge1$. Since the $A_5,\:S_4,\:A_4,\:D_{pn}$ cases have been explored, we shall discuss the case when $\mathcal{A}_{\alpha}$ is isomorphic to ${\mathbb Z}_{pn}$, $p\ge2$. There is only one possible relationship between $G$ and $\mathcal{A}_{\alpha}$:
$\mathcal{A}_{\alpha}=\langle z\mapsto e^{\frac{2\pi i}{pn}}z\rangle$
and it is obvious that
$$
C_{pn}(z)=C_n(z)\cup C_n(e^{\frac{1}{pn}2\pi i}z)\cup C_n(e^{\frac{2}{pn}2\pi i}z)\cup\cdots\cup C_n(e^{\frac{p-1}{pn}2\pi i}z).
$$

\begin{thm}\label{tZ_n}
For any finite subset $\alpha\subseteq S$, $|\alpha|\ge3$, its stabilizer $\mathcal{A}_{\alpha}$ equals $G$ ($\simeq Z_n)$ $(n\ge2)$ if and only if all of the three claims are true:
\begin{enumerate}
\item $\alpha$ is a union of certain elements in
$$
\{\{\infty\},\:\{0\}\}\cup\{C_n(z)|z\in{\mathbb C}^*\};
$$
\item $\alpha$ is NOT a union of certain elements in
$$
\{\{\infty\},\:\{0\}\}\cup\{C_{pn}(z)|z\in{\mathbb C}^*\},\:p\ge2;
$$
\item $\alpha$ is NOT in the icosahedral, octahedral, tetrahedral or dihedral case.
\end{enumerate}
\end{thm}
The discussion of the range of possible values for the component index in this case is complicated and left in Section \ref{Z_n}. The results are
\begin{thm}\label{cZ_n}
The existence of $\alpha$ such that $\mathcal{A}_{\alpha}\simeq \mathbb Z_n$, $n\ge2$ with arbitrary element in $\{0, 1, 2\}\times\mathbb N$ as its cyclic-$n$ component index is shown in the table below, along with all possible sizes of $\alpha$ in each circumstances.
$$
\begin{array}{c|c|c|c|c|c|c|c}
\mathrm{Component\:Index}
&(0, 1)&(0, 2)&(1, 1)&(1, 2)&(2, 1)&(2, 2)
&(\nu, k),\:k\ge3\\\hline
n\ge4&\mathrm{Nay}&\mathrm{Nay}
&\mathrm{Yea\:}(1+n)&\mathrm{Yea\:}(1+2n)
&\mathrm{Nay}&\mathrm{Nay}
&\mathrm{Yea\:}(\nu+nk)\\\hline
n=2, 3&\mathrm{Nay}&\mathrm{Nay}
&\mathrm{Nay}&\mathrm{Yea\:}(1+2n)
&\mathrm{Nay}&\mathrm{Nay}
&\mathrm{Yea\:}(\nu+2k)
\end{array}
$$
$$
\{|\alpha|<\infty\:|\:\mathcal{A}_{\alpha}\simeq \mathbb Z_n\}=
\{x+nk\:|\:x=0, 1, 2,\:k\ge3,\}\cup\{1+n, 1+2n\},\:n\ge4.
$$
$$
\{|\alpha|<\infty\:|\:\mathcal{A}_{\alpha}\simeq \mathbb Z_n\}=
\{x+nk\:|\:x=0, 1, 2,\:k\ge3,\}\cup\{1+2n\},\:n=2, 3.
$$
\end{thm}

\subsection{The Classification Algorithm}
\label{algorithm}

We need another theorem (proved in Subsection\ref{trivial}) to complete our classification algorithm:
\begin{thm}\label{ctrivial}
Suppose that $\mathcal{A}_{\alpha}$ is the trivial group and $\alpha$ is finite. Then the range of possible values for $|\alpha|$ is
$$
\{|\alpha|<\infty\:|\:\mathcal{A}_{\alpha}\:\mathrm{trivial}\}=
\{5,\:6,\:7\:\cdots\}.
$$
\end{thm}

From these results we obtain the algorithm for the classification. (The reader can find the code for this algorithm written in C programming language in Section \ref{code}). Input the cardinality of a finite subset ${\alpha}$ of the extended complex plane, the algorithm will output all possible group structure(s) of its stabilizer $\mathcal{A}_{\alpha}$, along with its component index(es). Notice that this is a classification for stabilizer of finite subsets of the extended complex plane, and only applies to the orbifold singularities of $\mathfrak{M_{0,\:n}}$ when $n\ge5$. (See the remark following Corollary \ref{34}).

\begin{verbatim}
n=input(|\alpha|);

if(n<3){output{infinity}};

k=\floor(n/60), r=n-60k;
if(r=0&k>0){output{A_5, (0, 0, 0, k)}};
if(r=12){output{A_5, (1, 0, 0, k)}};
if(r=20){output{A_5, (0, 1, 0, k)}};
if(r=30){output{A_5, (0, 0, 1, k)}};
if(r=32){output{A_5, (1, 1, 0, k)}};
if(r=42){output{A_5, (1, 0, 1, k)}};
if(r=50){output{A_5, (0, 1, 1, k)}};
if(r=2&k>0){output{A_5, (1, 1, 1, k-1)}};

k=\floor(n/24), r=n-24k;
if(r=0&k>0){output{S_4, (0, 0, 0, k)}};
if(r=6){output{S_4, (1, 0, 0, k)}};
if(r=8){output{S_4, (0, 1, 0, k)}};
if(r=12){output{S_4, (0, 0, 1, k)}};
if(r=14){output{S_4, (1, 1, 0, k)}};
if(r=18){output{S_4, (1, 0, 1, k)}};
if(r=20){output{S_4, (0, 1, 1, k)}};
if(r=2&k>0){output{S_4, (1, 1, 1, k-1)}};

k=\floor(n/12), r=n-12k;
if(r=0&k>0){output{A_4, (0, 0, k)}};
if(r=4){output{A_4, (1, 0, k)}};
if(r=8$k>0){output{A_4, (2, 0, k)}};
if(r=6&k>0){output{A_4, (0, 1, k)}};
if(r=10){output{A_4, (1, 1, k)}};
if(r=2&k>1){output{A_4, (2, 1, k-1)}};

p=n;
while(p>2)
begin
k=\floor(n/2p), l=\floor(n/p)-2k, r=n-2pk-pl;
if(k>0&r=0){output{D_p, (0, l, k)}};
if(k>1&l=0&r=0){output{D_p, (0, 2, k-1)}};
if(k>0&r=2){output{D_p, (1, l, k)}};
if(k>1&l=0&r=2){output{D_p, (1, 2, k-1)}};
if(k=0&l=1&r=0){output{D_p, (0, 1, 0)}};
if(k=0&l=1&r=2&p!=4){output{D_p, (1, 1, 0)}};
p=p-1;
end

k=\floor(n/4), r=n-4k;
if(k>0&r=0){output{K_4, (0, k)}};
if(k>1&r=0){output{K_4, (2, k-1)}};
if(k>0&r=2){output{K_4, (1, k)}};
if(k>1&r=2){output{K_4, (3, k-1)}};

p=n;
while(p>2)
begin
k=\floor(n/p), r=n-pk;
if(k>2&r<3){output{Z_p, (r, k)}};
if(k=2&r=1){output{Z_p, (r, k)}};
if(k=1&r=1&p!=3){output{Z_p, (r, k)}};
p=p-1;
end

k=\floor(n/2), r=n-2k;
if(k>2){output{Z_2, (r, k)}};
if(k>3&r=0){output{Z_2, (2, k-1)}};
if(k=2&r=1){output{Z_2, (r, k)}};

if(n>4){output{(0)}};
\end{verbatim}

There are some interesting observations:
\begin{thm}
For any finite group $G$ of linear fractional transformations, there exists a finite subset $\alpha\subseteq\widehat{\mathbb{C}}$, $|\alpha|\ge5$ such that $\mathcal{A}_{\alpha}=G$.
\end{thm}
\begin{cor}
For each finite subgroup $G$ of $\text{PSL}(2,\:\mathbb C)$, there exists some oribfold singularity of $\mathfrak{M_{0,\:n}}$ for some $n(\ge5)$ whose stabilizer is $G$.
\end{cor}

\section{The Moduli Space $\mathfrak{M_{0,\:n}}$}
\label{orbifold}

\subsection{The Orbifold Structure of Moduli Space $\mathfrak{M_{0,\:n}}$}
\label{orbif}

For $n\in\mathbb Z_{\ge4}$, set
$$
K_n=\{\boldsymbol{\lambda}=(\lambda^1,\:\lambda^2,\:\cdots,\:\lambda^{n-3})\in\mathbb C^{n-3}|\lambda^i\ne0,\:1,\:\lambda^i\ne\lambda^j,\:\forall i,\:j=1,\:2,\:\cdots,\:n-3,\:i\ne j\}.
$$
For any $\boldsymbol{\lambda}\in K_n$, set
$$
z^{\boldsymbol{\lambda}}_1=0,\:z^{\boldsymbol{\lambda}}_2=1,\:z^{\boldsymbol{\lambda}}_3=\infty,\:z^{\boldsymbol{\lambda}}_{i+3}=\lambda^{i},\:i=1,\:2,\:\cdots,\:n-3.
$$
For $\boldsymbol{\lambda}\in K_n$ and $\sigma\in S_n$, define $f^{\boldsymbol {\lambda}}_\sigma$ as the M\"obius transformation such that
$$
f^{\boldsymbol {\lambda}}_\sigma (z^{\boldsymbol{\lambda}}_{\sigma^{-1}(1)})=0,\:f^{\boldsymbol {\lambda}}_\sigma (z^{\boldsymbol{\lambda}}_{\sigma^{-1}(2)})=1,\:f^{\boldsymbol {\lambda}}_\sigma (z^{\boldsymbol{\lambda}}_{\sigma^{-1}(3)})=\infty.
$$
Now we define a mapping $g_\sigma$ of $K_n$ for each $\sigma \in S_n$
$$
g_\sigma :\:K_n\to K_n,\:\boldsymbol{\lambda}\mapsto(f^{\boldsymbol {\lambda}}_\sigma (z^{\boldsymbol{\lambda}}_{\sigma^{-1}(4)}),\:f^{\boldsymbol {\lambda}}_\sigma (z^{\boldsymbol{\lambda}}_{\sigma^{-1}(5)}),\:\cdots,\:f^{\boldsymbol {\lambda}}_\sigma (z^{\boldsymbol{\lambda}}_{\sigma^{-1}(n)})).
$$
By definition $g_\sigma$ is a rational map from $\mathbb C^{n-3}$ to $\mathbb C^{n-3}$ which restricts to an isomorphism of $K_n$. Here is a useful observation:
\begin{rem}
For any $\boldsymbol{\lambda}\in K_n$ and $\sigma\in S_n$, we have
$$
f^{\boldsymbol {\lambda}}_\sigma (z^{\boldsymbol{\lambda}}_k)=z^{g_\sigma(\boldsymbol{\lambda})}_{\sigma(k)}
$$
holds for $k=1,\:2,\:\cdots,\:n$.
\end{rem}
Define $G_n=\{g_\sigma|\sigma\in S_n\}$. Notice that $|G_n|\le n!$ and
\begin{thm}
$G_n$ is a finite group acting on $K_n$ when $n\ge4$.
\end{thm}
\begin{proof}
It is easy to check that
$$
f^{g_\sigma(\boldsymbol {\lambda})}_\pi\circ f^{\boldsymbol {\lambda}}_\sigma (z^{\boldsymbol{\lambda}}_{(\pi\cdot\sigma)^{-1}(k)})=f^{g_\sigma(\boldsymbol {\lambda})}_\pi(z^{g_\sigma(\boldsymbol{\lambda})}_{\pi^{-1}(k)})=z^{g_\pi\circ g_\sigma(\boldsymbol {\lambda})}_k,\:k=1,\:2,\:\cdots,\:n.
$$
Thus we conclude that
$$
f^{g_\sigma(\boldsymbol {\lambda})}_\pi\circ f^{\boldsymbol {\lambda}}_\sigma =f^{\boldsymbol {\lambda}}_{\pi\cdot\sigma},\:z^{g_\pi\circ g_\sigma(\boldsymbol {\lambda})}_k=z_k^{g_{\pi\cdot\sigma}(\boldsymbol {\lambda})},\:k=1,\:2,\:\cdots,\:n.
$$
So $g_\pi\circ g_\sigma=g_{\pi\cdot\sigma}\in G_n$. Thus $G_n$ is a group acting on $K_n$.
\end{proof}

As $G_n$ is a finite group acting on $K_n$, we have another conclusion
\begin{thm}
$\mathfrak{M_{0,\:n}}\simeq K_n/G_n$ as orbitfolds when $n\ge4$.
\end{thm}
\begin{proof}
Given two elements
$$
\overline{\{0,\:1,\:\infty,\:\lambda^1,\:\lambda^2,\:\cdots,\:\lambda^{n-3}\}},\:\overline{\{0,\:1,\:\infty,\:\mu^1,\:\mu^2,\:\cdots,\:\mu^{n-3}\}}\in \mathfrak{M_{0,\:n}},
$$
set
$$
\boldsymbol{\lambda}=(\lambda^1,\:\lambda^2,\:\cdots,\:\lambda^{n-3}),\:\boldsymbol{\mu}=(\mu^1,\:\mu^2,\:\cdots,\:\mu^{n-3}).
$$

If $\boldsymbol{\mu}\in G_n(\boldsymbol{\lambda})$, then there exists some $\sigma\in S_n$ s.t.~$\boldsymbol{\mu}=g_\sigma(\boldsymbol{\lambda}).$ So we have
$$
f^{\boldsymbol {\lambda}}_\sigma(\{0,\:1,\:\infty,\:\lambda^1,\:\lambda^2,\:\cdots,\:\lambda^{n-3}\})=\{0,\:1,\:\infty,\:\mu^1,\:\mu^2,\:\cdots,\:\mu^{n-3}\}.
$$
Thus
$$
\overline{\{0,\:1,\:\infty,\:\lambda^1,\:\lambda^2,\:\cdots,\:\lambda^{n-3}\}}=\overline{\{0,\:1,\:\infty,\:\mu^1,\:\mu^2,\:\cdots,\:\mu^{n-3}\}}.
$$

On the other hand, if
$$
\overline{\{0,\:1,\:\infty,\:\lambda^1,\:\lambda^2,\:\cdots,\:\lambda^{n-3}\}}=\overline{\{0,\:1,\:\infty,\:\mu^1,\:\mu^2,\:\cdots,\:\mu^{n-3}\}},
$$
then there exists some linear fractional transformation $h$ s.t.
$$
h(\{0,\:1,\:\infty,\:\lambda^1,\:\lambda^2,\:\cdots,\:\lambda^{n-3}\})=\{0,\:1,\:\infty,\:\mu^1,\:\mu^2,\:\cdots,\:\mu^{n-3}\}
$$
and some $\sigma\in S_n$ s.t.~
$$
h(z^{\boldsymbol{\lambda}}_{\sigma^{-1}(k)})=z^{\boldsymbol{\mu}}_k,\:k=1,\:2,\:\cdots,\:n.
$$
So $h=f^{\boldsymbol {\lambda}}_\sigma$ and $\boldsymbol{\mu}=g_\sigma(\boldsymbol{\lambda})$.

\end{proof}

\begin{cor}
$\mathfrak{M_{0,\:n}}$ is a complex orbifold of dimension $n-3$ when $n\ge4$.
\end{cor}

\subsection{The Explicit Structure of the Group $G_n$}
\label{grou}

Now let us explore $G_n$. Define subgroup $V_n$ of the symmetric group $S_n$
\begin{displaymath}
V_n=S_{\{1,\:2,\:3\}}\times S_{\{4,\:5,\:\cdots,\:n\}}
\end{displaymath}
and subset $T_n$ of $S_n$
\[
\begin{split}
T_n=\{&e,\:(p,\:1),\:(p,\:2),\:(p,\:3),\:(p,\:1)(q,\:2),\:(p,\:2)(q,\:3),\:(p,\:3)(q,\:1),\:(p,\:1)(q,\:2)(r,\:3)|\\
    &p,\:q,\:r=4,\:5,\:\cdots,\:n,\:p\ne q,\:q\ne r,\:r\ne p\}.
\end{split}
\]
Thus $T_n$ forms a complete list of left coset representatives of $V_n\subseteq S_n$.

Define group $H$ of linear fractional transformations
\begin{displaymath}
H=\{\mathrm{Id},\:\lambda\mapsto 1-\lambda,\: \lambda\mapsto \frac{1}{\lambda},\:\lambda\mapsto\frac{\lambda}{\lambda-1} ,\:\lambda\mapsto \frac{\lambda-1}{\lambda},\:\lambda\mapsto \frac{-1}{\lambda-1}\}.
\end{displaymath}
It is obvious that
\begin{displaymath}
\{g_\sigma|\sigma\in V_n\}=\{\boldsymbol{\lambda}\mapsto(h(\lambda^{\tau(1)}),\:h(\lambda^{\tau(2)}),\:\cdots,\:h(\lambda^{\tau(n-3)}))|h\in H,\:\tau\in S_{n-3}\}.
\end{displaymath}

For $p=4,\:5,\:\cdots,\:n$, set $\sigma=(1,\:p)$ and we have
\begin{displaymath}
f^{\boldsymbol {\lambda}}_{(1,\:p)}(\lambda^{p-3})=
f^{\boldsymbol {\lambda}}_{(1,\:p)} (z^{\boldsymbol{\lambda}}_{\sigma^{-1}(1)})=0,
\end{displaymath}
\begin{displaymath}
f^{\boldsymbol {\lambda}}_{(1,\:p)}(1)=
f^{\boldsymbol {\lambda}}_{(1,\:p)}(z^{\boldsymbol{\lambda}}_{\sigma^{-1}(2)})=1,
\end{displaymath}
\begin{displaymath}
f^{\boldsymbol {\lambda}}_{(1,\:p)}(\infty)=
f^{\boldsymbol {\lambda}}_{(1,\:p)}(z^{\boldsymbol{\lambda}}_{\sigma^{-1}(3)})=\infty.
\end{displaymath}
So we have
\begin{displaymath}
f^{\boldsymbol {\lambda}}_{(1,\:p)}(z)=\frac{z-\lambda^{p-3}}{1-\lambda^{p-3}},
\end{displaymath}
and
$$
g_{(1,\:p)}^{(k)}(\boldsymbol {\lambda})=
\begin{dcases}
\frac{\lambda^k-\lambda^{p-3}}{1-\lambda^{p-3}}, &k\ne p-3; \\
\frac{-\lambda^{p-3}}{1-\lambda^{p-3}},          &k=p-3.
\end{dcases}
$$

By the same method I calculated $g_\sigma$ for each $\sigma\in S_n$. The next theorem is the result.
\begin{thm}
$G_n=\tilde T_n\tilde V_n$, where
$$
\tilde V_n=\{\boldsymbol{\lambda}\mapsto(h(\lambda^{\tau(1)}),\:h(\lambda^{\tau(2)}),\:\cdots,\:h(\lambda^{\tau(n-3)}))|h\in H,\:\tau\in S_{n-3}\};
$$
$$
H=\{\mathrm{Id},\:\lambda\mapsto 1-\lambda,\: \lambda\mapsto \frac{1}{\lambda},\:\lambda\mapsto\frac{\lambda}{\lambda-1} ,\:\lambda\mapsto \frac{\lambda-1}{\lambda},\:\lambda\mapsto \frac{-1}{\lambda-1}\};
$$
and
\[
\begin{split}
\tilde T_n=\{g_\sigma|&\sigma=e,\:(p,\:1),\:(p,\:2),\:(p,\:3),\:(p,\:1)(q,\:2),\:(p,\:2)(q,\:3),\:(p,\:3)(q,\:1),\:(p,\:1)(q,\:2)(r,\:3)\\
&p,\:q,\:r=4,\:5,\:\cdots,\:n,\:p\ne q,\:q\ne r,\:r\ne p\};
\end{split}
\]
$$
g_{(1,\:p)}^{(k)}(\boldsymbol {\lambda})=
\begin{dcases}
\frac{\lambda^k-\lambda^{p-3}}{1-\lambda^{p-3}}, &k\ne p-3; \\
\frac{-\lambda^{p-3}}{1-\lambda^{p-3}},          &k=p-3;
\end{dcases}
$$
$$
g_{(2,\:p)}^{(k)}(\boldsymbol {\lambda})=
\begin{dcases}
\frac{\lambda^k}{\lambda^{p-3}}, &k\ne p-3; \\
\frac{1}{\lambda^{p-3}},         &k=p-3;
\end{dcases}
$$
$$
g_{(3,\:p)}^{(k)}(\boldsymbol {\lambda})=
\begin{dcases}
\frac{\lambda^k(1-\lambda^{p-3})}{\lambda^k-\lambda^{p-3}}, &k\ne p-3; \\
\frac{\lambda^{p-3}-1}{\lambda^{p-3}},                      &k=p-3;
\end{dcases}
$$
$$
g_{(1,\:p)(2\:,q)}^{(k)}(\boldsymbol {\lambda})=
\begin{dcases}
\frac{\lambda^k-\lambda^{p-3}}{\lambda^{q-3}-\lambda^{p-3}}, &k\ne p-3,\:q-3;\\
\frac{-\lambda^{p-3}}{\lambda^{q-3}-\lambda^{p-3}},          &k=p-3;\\
\frac{1-\lambda^{p-3}}{\lambda^{q-3}-\lambda^{p-3}},         &k=q-3;
\end{dcases}
$$
$$
g_{(2,\:p)(3\:,q)}^{(k)}(\boldsymbol {\lambda})=
\begin{dcases}
\frac{(\lambda^{p-3}-\lambda^{q-3})\lambda^k}{\lambda^{p-3}(\lambda^k-\lambda^{q-3})},&k\ne p-3,\:q-3;\\
\frac{\lambda^{p-3}-\lambda^{q-3}}{\lambda^{p-3}(1-\lambda^{q-3})},&k=p-3;\\
\frac{\lambda^{p-3}-\lambda^{q-3}}{\lambda^{p-3}},&k=q-3;
\end{dcases}
$$
$$
g_{(3,\:p)(1\:,q)}^{(k)}(\boldsymbol {\lambda})=
\begin{dcases}
\frac{(\lambda^{p-3}-1)(\lambda^k-\lambda^{q-3})}{(\lambda^{q-3}-1)(\lambda^k-\lambda^{p-3})},&k\ne p-3,\:q-3;\\
\frac{\lambda^{p-3}-1}{\lambda^{q-3}-1},&k=p-3;\\
\frac{(\lambda^{p-3}-1)\lambda^{q-3}}{(\lambda^{q-3}-1)\lambda^{p-3}},&k=q-3;
\end{dcases}
$$
$$
g_{(1,\:p)(2\:,q)(3,\:r)}^{(k)}(\boldsymbol {\lambda})=
\begin{dcases}
\frac{(\lambda^{q-3}-\lambda^{r-3})(\lambda^k-\lambda^{p-3})}{(\lambda^{q-3}-\lambda^{p-3})(\lambda^k-\lambda^{r-3})}, &k\ne p-3,\:q-3,\:r-3;\\
\frac{(\lambda^{q-3}-\lambda^{r-3})\lambda^{p-3}}{(\lambda^{q-3}-\lambda^{p-3})\lambda^{r-3}},&k=p-3;\\
\frac{(\lambda^{q-3}-\lambda^{r-3})(1-\lambda^{p-3})}{(\lambda^{q-3}-\lambda^{p-3})(1-\lambda^{r-3})},&k=q-3;\\
\frac{\lambda^{q-3}-\lambda^{r-3}}{\lambda^{q-3}-\lambda^{p-3}},&k=r-3.
\end{dcases}
$$
\end{thm}

\begin{cor}
$G_n$ is isomorphic to $S_n$ when $
{n\ge5}$.
\end{cor}

\subsection{The Singularities of Moduli Space $\mathfrak{M_{0,\:n}}$ }
\label{sing}

For $\boldsymbol {\lambda}\in K_n$, $n\ge5$, set
$$
[\boldsymbol {\lambda}]=\{0,\:1,\:\infty,\:\lambda^1,\:\lambda^2,\:\cdots,\:\lambda^{n-3}\}=\{z^{\boldsymbol {\lambda}}_k|k=1,\:2,\:\cdots,\:n\}
$$
and $G_{\boldsymbol \lambda}=\{g_{\sigma}\in G_n\:|\:g_\sigma(\boldsymbol \lambda)=\boldsymbol \lambda\}$ the stabilizer of $\boldsymbol {\lambda}$
\begin{defn}
\label{def}
When $n\ge5$, call $\overline{[\boldsymbol {\lambda}]}\in \mathfrak{M_{0,\:n}}$ an oribfold singularity of $\mathfrak{M_{0,\:n}}$ if $G_{\boldsymbol \lambda}$ is non-trivial.
\end{defn}

Given a finite subset $\alpha=\{z_1, z_2, \cdots, z_n\}\subseteq \widehat{\mathbb{C}}$, $n\ge4$, let $\mathcal{A}_{\alpha}$ denote the stabilizing group, i.e., the group of linear fractional transformations that fix $\alpha$. It is easy to check that


\begin{thm}
The mapping
$$
\Phi_{\boldsymbol \lambda}:\:G_{\boldsymbol \lambda}\to\mathcal{A}_{[\boldsymbol {\lambda}]},\: g_\sigma\mapsto f^{\boldsymbol {\lambda}}_\sigma
$$
is a group isomorphism 
for any $\boldsymbol {\lambda}\in K_{n}$, $n\ge5$.
\end{thm}

\section{Possible Values for the Cardinality of $\alpha$}
\label{cardinality}

\subsection{Dihedral Group}
\label{D_n}

First suppose that $n\ge3$. From Theorem \ref{tD_n} and observation we know that there exist some finite subsets $\alpha,\:\beta\subseteq S^2$ such that $\mathcal{A}_{\alpha}\simeq D_n$ with dihedral-$n$ component index $(0, 1, 0)$, and $\mathcal{A}_{\beta}\simeq D_n$ (if $n\ne4$) with dihedral-$n$ component index $(1, 1, 0)$.
When the dihedral-$n$ component index of $\alpha$ is $(0, 2, 0)$ or $(1, 2, 0)$, $\mathcal{A}_{\alpha}$ is isomorphic to $D_{2n}$.
When $n=4$ and the dihedral-$4$ component index of $\beta$ is $(1, 1, 0)$,  $\mathcal{A}_{\beta}$ is isomorphic to $S_4$.

Then suppose that $n=2$. It is obvious that when the dihedral-$2$  component index of $\alpha$ is $(1, 0)$, $(2, 0)$ and $(3, 0)$, $\mathcal{A}_{\alpha}$ is infinite, isomorphic to $D_4$ and $S_4$, respectively.

Now suppose that $k\ge1$. In the proposition below, I shall give examples of $\alpha$ such that $\mathcal{A}_{\alpha}\simeq D_n$ with dihedral-$n$  component index $(0, 0, k)$ when $n\ge3$, and $\alpha$ such that $\mathcal{A}_{\alpha}\simeq K_4$ with dihedral-$2$ component index $(0, k)$. Then it will be easy to find examples of $\alpha$ with dihedral-$n$ component index $(\nu, \epsilon, k)$ or $(\nu, k)$ for each element $(\nu, \epsilon, k)\in\{0, 1\}\times\{0, 1, 2\}\times\mathbb N^+$ and $(\nu, k)\in\{0, 1, 2, 3\}\times\mathbb N^+$ by the same means. Thus the proof of Theorem \ref{cD_n} is completed.

\begin{prop}
The finite subset
$$
\alpha=\bigcup_{l=1}^k C_n(e^{\frac{l}{8k^2n}2\pi i}),
$$
where $C_n(z)=\{ze^{\frac{k}{n}2\pi i}|k\in\mathbb{Z}\}\cup\{z^{-1}e^{\frac{k}{n}2\pi i}|k\in\mathbb{Z}\}$ and $k\ge1$,
has stabilizer $\mathcal{A}_{\alpha}\simeq D_n$ with dihedral-$n$ component index $(0,\:0,\:k)$ when $n\ge3$ and stabilizer $\mathcal{A}_{\alpha}\simeq K_4$ with dihedral-$2$ component index $(0,\:k)$ when $n=2$.
\end{prop}
\begin{proof}
It is obvious that
$$
D_n\simeq\langle z\mapsto e^{\frac{1}{n}2\pi i}z,\:
          z\mapsto \frac{1}{z}              \rangle
\subseteq\mathcal{A}_{\alpha}.
$$
$\mathcal{A}_{\alpha}$ is not isomorphic to $A_5$, $S_4$ or $A_4$ as all points of $\alpha$ are concyclic. Assume that
$$
\mathcal{A}_{\alpha}
=\langle \rho,\pi|(\rho)^{pn}=(\pi)^2=(\rho\pi)^2=e\rangle
\simeq D_{pn}
$$
for some $2\le p\le 2k$.

\emph{Claim:} There exists some $m\in\mathbb Z$ such that $\rho^m(z)=ze^{\frac{1}{n}2\pi i}$.

\emph{Proof of Claim:} The claim is obvious when $n\ge3$. When $n=2$, it is clear that
$$
\rho^p\in\langle z\mapsto-z,\:z\mapsto\frac{1}{z}\rangle\simeq K_4.
$$
Suppose that $\rho^p(z)=\frac{1}{z}$. Because $\rho$ is a M\"obius transformation of order $2p\ge4$, it is elliptic with fixed points $1$ and $-1$, and thus never fixes $\alpha$. Thus $\rho^p(z)\ne\frac{1}{z}$. For the same reason $\rho^p(z)\ne\frac{-1}{z}$. Thus $\rho^p(z)=-z$, and the claim is proved.

Because $\rho$ is a M\"obius transformation of order $pn$, it is elliptic with fixed points $0$ and $\infty$. Thus
$$
\rho(z)=ze^{\frac{n'}{pn}2\pi i},\:(pn, n')=1.
$$
Without lose of generality assume that $n'=1$. Then
$$
\rho(e^{\frac{1}{8k^2n}2\pi i})
=e^{(\frac{1}{8k^2n}+\frac{1}{pn})2\pi i}\in\alpha
$$
with
$$
\frac{1}{8k^2n}+\frac{1}{pn}\in
[\frac{1}{8k^2n}+\frac{1}{2kn},\:\frac{1}{8k^2}+\frac{1}{2n}]
$$
since $2\le p\le 2k$. However, by the construction of $\alpha$ we know that
$$
\alpha\cap\{e^{t2\pi i}\:|\:t\in
(\frac{1}{8kn},\:\frac{1}{n}-\frac{1}{8kn})\}=\emptyset.
$$
Thus
$$
\frac{1}{8k^2n}+\frac{1}{pn}\in
[\frac{1}{n}-\frac{1}{8kn},\:\frac{1}{8k^2}+\frac{1}{2n}],
$$
which contradicts the assumption that $p\ge2$.
\end{proof}

\subsection{Cyclic Group}
\label{Z_n}

Suppose that $n\ge3$. Set
$$
\alpha_{n, k}=\bigcup^k_{l=1}C_n(l),\:
\beta_{n, k}=\alpha\cup\{0\},\:
\gamma_{n, k}=\beta\cup\{\infty\}
$$
where $C_n(z)=\{ze^{\frac{k}{n}2\pi i}|k\in\mathbb{Z}\}$. From Theorem \ref{tZ_n} and observation we know that when $k\ge3$
$$
\mathcal{A}_{\alpha_{n, k}}=\mathcal{A}_{\gamma_{n, k}}=
\langle z\mapsto e^{\frac{1}{n}2\pi i}z\rangle\simeq\mathbb Z_n
$$
and when $k\ge1$ (except for $\beta_{3, 1}$),
$$
\mathcal{A}_{\beta_{n, k}}=
\langle z\mapsto e^{\frac{1}{n}2\pi i}z\rangle\simeq\mathbb Z_n.
$$
Thus we have found subsets whose stabilizers isomorphic to $\mathbb Z_n$ with cyclic-$n$ component index $(0, k)$ and $(2, k)$ for $k\ge3$ and cyclic-$n$ component index $(1, k)$ for $k\ge1$ (except for $(1, 1)$ when $n=3$.) When the cyclic-$n$ component index of $\alpha$ is $(0, 1)$, $(0, 2)$, $(2, 1)$ or $(2, 2)$, $\mathcal{A}_{\alpha}$ is isomorphic to $D_{n}$. When $n=3$ and the cyclic-$3$ component index of $\alpha$ is $(1, 1)$, $\mathcal{A}_{\alpha}$ is isomorphic to $A_4$.

Now suppose that $n=2$. In the following proposition we will find finite subsets whose stabilizers are isomorphic to $\mathbb Z_2$ with cyclic-$2$ component index $(1, k)$ for each $k\ge2$.
\begin{prop}
The finite subset
$$
\alpha=\{0, \pm1, \pm2, \cdots \pm k\},\:k\ge2
$$
has stabilizer
$\mathcal{A}_{\alpha}=\langle z\mapsto-z\rangle\simeq\mathbb Z_2$ with cyclic-$2$ component index $(1, k)$.
\end{prop}
\begin{proof}
$\mathcal{A}_{\alpha}$ is not isomorphic to $A_5$, $S_4$ or $A_4$ as all points of $\alpha$ are concyclic.

Assume that
$$
\mathcal{A}_{\alpha}=\langle\rho\rangle\simeq\mathbb Z_{2p}
$$
for some $p\ge2$. Then $\rho$ is an elliptic transformation such that $\rho^p(z)=-z$. Thus
$$
\rho(z)=ze^{\frac{n'}{2p}2\pi i},\:(2p, n')=1.
$$
Thus $\rho$ can never fix $\alpha$ and the assumption is wrong.

Assume that
$$
\mathcal{A}_{\alpha}
=\langle \rho,\pi|(\rho)^{n}=(\pi)^2=(\rho\pi)^2=e\rangle
\simeq D_{n}
$$
for some $n\ge 2$. Notice that all points of $\alpha$ are concyclic. Then from Theorem \ref{tD_n} we know that $n|2k+1$ (thus $n\ge3$), and there exists some $\tilde\rho=\rho^m$ such that
$$
\tilde\rho(0)=\frac{2k+1}{n},\:
\tilde\rho(-1)=-1+\frac{2k+1}{n},\:
\tilde\rho(-2)=-2+\frac{2k+1}{n}\le k.
$$
We conclude that
$$
\tilde\rho(z)=z+\frac{2k+1}{n}.
$$
Thus $\tilde\rho$ can never fix $\alpha$ and the assumption is wrong again.
\end{proof}

Thus we have found finite subsets whose stabilizers are isomorphic to $\mathbb Z_2$ with cyclic-$2$ component index $(1, k)$ for each $k\ge2$. By the same method it is easy to find subsets whose stabilizers isomorphic to $\mathbb Z_2$ with cyclic-$2$ component index $(0, k)$ and $(2, k)$ for $k\ge3$ (consider $\{\pm\frac{1}{2}, \pm\frac{3}{2}, \cdots \pm\frac{2k-1}{2}, \}$ and $\{0, \infty, \pm1, \pm2, \cdots \pm k\}$). Notice that when the cyclic-$2$ component index is $(0, 1)$, $(0, 2)$, $(1, 1)$, $(2, 1)$ or $(2, 2)$, $\mathcal{A}_{\alpha}$ is not isomorphic to $\mathbb Z_2$. Now the proof of Theorem \ref{cZ_n} is completed.

\subsection{Trivial Group}
\label{trivial}

Suppose that $\mathcal{A}_{\alpha}$ is the trivial group, $\alpha$ finite, what could $|\alpha|$ be? When $|\alpha|=2$, $\mathcal{A}_{\alpha}$ is infinite. When $|\alpha|=3$, $\mathcal{A}_{\alpha}$ is isomorphic to $D_3$. When $|\alpha|=4$, suppose that $\alpha=\{z_1, z_2, z_3, z_4\}$. It is easy to check that
$$
[z_1, z_2, z_3, z_4]=[z_2, z_1, z_4, z_3]=
[z_3, z_4, z_1, z_2]=[z_4, z_3, z_2, z_1].
$$
Thus $\mathcal{A}_{\alpha}$ can not be trivial.
\begin{prop}
The finite subset
$$
\alpha=C_n(1)\cup\{2\},
$$
where $C_n(z)=\{ze^{\frac{k}{n}2\pi i}|k\in\mathbb{Z}\}$
has a trivial stabilizing group when $n\ge4$.
\end{prop}
\begin{proof}
For any $f\in\mathcal{A}_{\alpha}$,
$$
|f(C_n(1))\cap C_n(1)|\ge n-1\ge3.
$$
Thus $f$ fixes the unit circle, which implies that $f$ also fixes $C_n(1)$. From Theorem \ref{tD_n} we know that
$$
f\in\langle z\mapsto e^{\frac{1}{n}2\pi i}z,\:
            z\mapsto \frac{1}{z}              \rangle.
$$
However $f$ also fixes $2$. Thus $f$ has to be the identity.
\end{proof}
Thus we have proved Theorem \ref{ctrivial}.

\section{The Code of the Algorithm}
\label{code}

\begin{verbatim}
#include <stdio.h>

int main()
{int n, k, l, r, p;
printf("Please enter the size of the subset.\n");
scanf("%d", &n);

if(n<1){printf("error");}
if(n==1){printf("infinity");}
if(n==2){printf("infinity");}


k=n/60; r=n-60*k;
//printf("k=%d, r= %d\n",k,r);
if(r==0&k>=1){printf("A_5, (0, 0, 0, %d)\n",k);}
if(r==12){printf("A_5, (1, 0, 0, %d)\n",k);}
if(r==20){printf("A_5, (0, 1, 0, %d)\n",k);}
if(r==30){printf("A_5, (0, 0, 1, %d)\n",k);}
if(r==32){printf("A_5, (1, 1, 0, %d)\n",k);}
if(r==42){printf("A_5, (1, 0, 1, %d)\n",k);}
if(r==50){printf("A_5, (0, 1, 1, %d)\n",k);}
if(r==2&k>=1){printf("A_5, (1, 1, 1, %d)\n",k-1);}

k=n/24; r=n-24*k;
//printf("k=%d, r= %d\n",k,r);
if(r==0&k>=1){printf("S_4, (0, 0, 0, %d)\n",k);}
if(r==6){printf("S_4, (1, 0, 0, %d)\n",k);}
if(r==8){printf("S_4, (0, 1, 0, %d)\n",k);}
if(r==12){printf("S_4, (0, 0, 1, %d)\n",k);}
if(r==14){printf("S_4, (1, 1, 0, %d)\n",k);}
if(r==18){printf("S_4, (1, 0, 1, %d)\n",k);}
if(r==20){printf("S_4, (0, 1, 1, %d)\n",k);}
if(r==2&k>=1){printf("S_4, (1, 1, 1, %d)\n",k-1);}

k=n/12, r=n-12*k;
//printf("k=%d, r= %d\n",k,r);
if(r==0&k>=1){printf("A_4, (0, 0, %d)\n",k);}
if(r==4){printf("A_4, (1, 0, %d)\n",k);}
if(r==8&k>=1){printf("A_4, (2, 0, %d)\n",k);}
if(r==6&k>=1){printf("A_4, (0, 1, %d)\n",k);}
if(r==10){printf("A_4, (1, 1, %d)\n",k);}
if(r==2&k>=2){printf("A_4, (2, 1, %d)\n",k-1);}

p=n;
while(p>=3)
{
k=n/(2*p); l=n/p-2*k; r=n-2*p*k-p*l;
//printf("p=%d, k=%d, l=%d, r= %d\n",p,k,l,r);
if(k>=1&r==0){printf("D_%d, (0, %d, %d)\n",p,l,k);}
if(k>=2&r==0&l==0){printf("D_%d, (0, 2, %d)\n",p,k-1);}
if(k>=1&r==2){printf("D_%d, (1, %d, %d)\n",p,l,k);}
if(k>=2&r==2&l==0){printf("D_%d, (1, 2, %d)\n",p,k-1);}
if(k==0&r==0&l==1){printf("D_%d, (0, 1, 0)\n",p);}
if(k==0&r==2&l==1&p!=4){printf("D_%d, (1, 1, 0)\n",p);}
p=p-1;
}

k=n/4; r=n-4*k;
//printf("k=%d, r= %d\n",k,r);
if(k>=1&r==0){printf("K_4, (0, %d)\n",k);}
if(k>=2&r==0){printf("K_4, (2, %d)\n",k-1);}
if(k>=1&r==2){printf("K_4, (1, %d)\n",k);}
if(k>=2&r==2){printf("K_4, (3, %d)\n",k-1);}

p=n;
while(p>=3)
{
k=n/p; r=n-p*k;
//printf("p=%d, k=%d, r= %d\n",p,k,r);
if(k>=3&r<=2){printf("Z_%d, (%d, %d)\n",p,r,k);}
if(k==2&r==1){printf("Z_%d, (%d, %d)\n",p,r,k);}
if(k==1&r==1&p!=3){printf("Z_%d, (%d, %d)\n",p,r,k);}
p=p-1;
}

k=n/2; r=n-2*k;
//printf("k=%d, r= %d\n",k,r);
if(k>=3){printf("Z_2, (%d, %d)\n",r,k);}
if(k>=4&r==0){printf("Z_2, (2, %d)\n",k-1);}
if(k==2&r==1){printf("Z_2, (%d, %d)\n",r,k);}

if(n>=5){printf("(0)");}

return 0; }
\end{verbatim}

\begin{examp}
The classification of the obifold singularities of the moduli space $\mathfrak{M_{0,\:2018}}$ is
\begin{verbatim}
S_4, (1, 1, 1, 83)
A_4, (2, 1, 167)
D_2018, (0, 1, 0)
D_2016, (1, 1, 0)
D_1009, (0, 0, 1)
D_1008, (1, 0, 1)
D_672, (1, 1, 1)
D_504, (1, 0, 2)
D_504, (1, 2, 1)
D_336, (1, 0, 3)
D_336, (1, 2, 2)
D_288, (1, 1, 3)
D_252, (1, 0, 4)
D_252, (1, 2, 3)
D_224, (1, 1, 4)
D_168, (1, 0, 6)
D_168, (1, 2, 5)
D_144, (1, 0, 7)
D_144, (1, 2, 6)
D_126, (1, 0, 8)
D_126, (1, 2, 7)
D_112, (1, 0, 9)
D_112, (1, 2, 8)
D_96, (1, 1, 10)
D_84, (1, 0, 12)
D_84, (1, 2, 11)
D_72, (1, 0, 14)
D_72, (1, 2, 13)
D_63, (1, 0, 16)
D_63, (1, 2, 15)
D_56, (1, 0, 18)
D_56, (1, 2, 17)
D_48, (1, 0, 21)
D_48, (1, 2, 20)
D_42, (1, 0, 24)
D_42, (1, 2, 23)
D_36, (1, 0, 28)
D_36, (1, 2, 27)
D_32, (1, 1, 31)
D_28, (1, 0, 36)
D_28, (1, 2, 35)
D_24, (1, 0, 42)
D_24, (1, 2, 41)
D_21, (1, 0, 48)
D_21, (1, 2, 47)
D_18, (1, 0, 56)
D_18, (1, 2, 55)
D_16, (1, 0, 63)
D_16, (1, 2, 62)
D_14, (1, 0, 72)
D_14, (1, 2, 71)
D_12, (1, 0, 84)
D_12, (1, 2, 83)
D_9, (1, 0, 112)
D_9, (1, 2, 111)
D_8, (1, 0, 126)
D_8, (1, 2, 125)
D_7, (1, 0, 144)
D_7, (1, 2, 143)
D_6, (1, 0, 168)
D_6, (1, 2, 167)
D_4, (1, 0, 252)
D_4, (1, 2, 251)
D_3, (1, 0, 336)
D_3, (1, 2, 335)
K_4, (1, 504)
K_4, (3, 503)
Z_2017, (1, 1)
Z_672, (2, 3)
Z_504, (2, 4)
Z_336, (2, 6)
Z_288, (2, 7)
Z_252, (2, 8)
Z_224, (2, 9)
Z_168, (2, 12)
Z_144, (2, 14)
Z_126, (2, 16)
Z_112, (2, 18)
Z_96, (2, 21)
Z_84, (2, 24)
Z_72, (2, 28)
Z_63, (2, 32)
Z_56, (2, 36)
Z_48, (2, 42)
Z_42, (2, 48)
Z_36, (2, 56)
Z_32, (2, 63)
Z_28, (2, 72)
Z_24, (2, 84)
Z_21, (2, 96)
Z_18, (2, 112)
Z_16, (2, 126)
Z_14, (2, 144)
Z_12, (2, 168)
Z_9, (2, 224)
Z_8, (2, 252)
Z_7, (2, 288)
Z_6, (2, 336)
Z_4, (2, 504)
Z_3, (2, 672)
Z_2, (0, 1009)
Z_2, (2, 1008)
(0)
\end{verbatim}
\end{examp}

\bibliographystyle{plain}

\small{\noindent YUE WU\\
SCHOOL OF MATHEMATICAL SCIENCES\\
UNIVERSITY OF SCIENCE AND TECHNOLOGY OF CHINA\\
HEFEI 230026 CHINA\\
wuyuee15@mail.ustc.edu.cn}\\

\small{\noindent BIN XU\\
WU WEN-TSUN KEY LABORATORY OF MATH, USTC, CHINESE ACADEMY OF SCIENCE\\
SCHOOL OF MATHEMATICAL SCIENCES\\
UNIVERSITY OF SCIENCE AND TECHNOLOGY OF CHINA\\
HEFEI 230026 CHINA\\
bxu@ustc.edu.cn}

\end{document}